\documentclass{amsart}

\usepackage[T1]{fontenc}
\usepackage[utf8]{inputenc}
\usepackage{amsmath,amssymb}
\usepackage{aliascnt}
\usepackage{xurl}
\usepackage{hyperref}
\usepackage[nameinlink,noabbrev]{cleveref}
\hypersetup{hidelinks,
  pdftitle={A separably representable counterexample to Naimark's problem in ZFC},
  pdfauthor={Ryotaro Tanaka},
  pdfsubject={Naimark's problem and representations of C*-algebras},
  pdfkeywords={Naimark's problem, C*-algebras, irreducible representations,
    pure states, tracial states}}

\numberwithin{equation}{section}
\newtheorem{theorem}{Theorem}[section]
\newtheorem*{rosenbergtheorem}{Rosenberg's theorem}
\newaliascnt{lemma}{theorem}
\newtheorem{lemma}[lemma]{Lemma}
\aliascntresetthe{lemma}
\newaliascnt{proposition}{theorem}
\newtheorem{proposition}[proposition]{Proposition}
\aliascntresetthe{proposition}
\newaliascnt{corollary}{theorem}
\newtheorem{corollary}[corollary]{Corollary}
\aliascntresetthe{corollary}
\crefname{theorem}{Theorem}{Theorems}
\Crefname{theorem}{Theorem}{Theorems}
\crefname{lemma}{Lemma}{Lemmas}
\Crefname{lemma}{Lemma}{Lemmas}
\crefname{proposition}{Proposition}{Propositions}
\Crefname{proposition}{Proposition}{Propositions}
\crefname{corollary}{Corollary}{Corollaries}
\Crefname{corollary}{Corollary}{Corollaries}

\newcommand{\C}{\mathbb C}
\newcommand{\N}{\mathbb N}

\newcommand{\B}{\mathcal B}
\newcommand{\K}{\mathcal K}
\newcommand{\U}{\mathcal U}

\newcommand{\ip}[2]{\langle #1,#2\rangle}
\newcommand{\norm}[1]{\lVert #1\rVert}

\DeclareMathOperator{\ran}{ran}

\DeclareMathOperator{\id}{id}
\newcommand{\stlim}{\mathop{s\mathchar`-\mathrm{lim}}\displaylimits}
\newcommand{\rt}{\mathrm r}

\title[Naimark's problem in ZFC]{A separably representable counterexample
  to Naimark's problem in ZFC}
\author[R. Tanaka]{Ryotaro Tanaka}
\address[R. Tanaka]{Katsushika Division, Institute of Arts and Sciences,
Tokyo University of Science, Tokyo 125-8585, Japan}
\email{r-tanaka@rs.tus.ac.jp}
\thanks{The author was supported by JSPS KAKENHI Grant Number JP24K06788.}

\date{}

\begin{document}
\begin{abstract}
We construct in ZFC a unital, simple, infinite-dimensional $C^*$-algebra
whose nonzero irreducible representations form a single unitary equivalence
class, but which admits a faithful representation on a separable Hilbert
space. The algebra contains a unital copy of the canonical anticommutation
relation (CAR) algebra, and the normalized CAR trace has a unique extension
among all states. This extension
is tracial, and its GNS representation is separable and faithful, with weak
closure the hyperfinite $\mathrm{II}_1$ factor. In contrast, every nonzero
irreducible representation acts on a Hilbert space of density $2^{\aleph_0}$.
The construction separates the added unitaries into shell terms, controlled
by finite CAR relations, and rank-one defect terms. These relations determine
every irreducible representation of the generated algebra and every extension
of the CAR trace. The algebra has norm density $2^{\aleph_0}$; consequently,
the continuum hypothesis (CH) is equivalent over ZFC to the existence of a
counterexample to Naimark's problem of norm density $\aleph_1$.
\end{abstract}
\maketitle

\section{Introduction}\label{sec:main}

Naimark's problem asks whether a $C^*$-algebra whose nonzero irreducible
representations are all unitarily equivalent must be isomorphic to the
algebra $\K(K)$ of compact operators on some complex Hilbert space $K$
\cite{Naimark1951}.

A $C^*$-algebra is \emph{separably representable} if it admits a faithful
representation on a separable Hilbert space. The representation need not
be irreducible, nor the algebra norm separable. The relevant classical
obstruction is the following.

\begin{rosenbergtheorem}[{\cite[Theorem~4, p.~530]{Rosenberg}}]
Let $K$ be a nonzero separable Hilbert space, and let
$B\subseteq\B(K)$ be a nonzero $C^*$-algebra acting irreducibly on $K$.
If all nonzero irreducible representations of $B$ are unitarily equivalent,
then $B=\K(K)$.
\end{rosenbergtheorem}

Thus a counterexample cannot have a separable irreducible representation.
We construct one in ZFC with a faithful separable tracial representation,
and with unique state extension from a specified CAR subalgebra.

\begin{theorem}\label{thm:main}
There exists a unital, simple, infinite-dimensional $C^*$-algebra $A$ with exactly one unitary equivalence class of nonzero irreducible representations and a unique tracial state $\widetilde\tau$. The GNS representation of $\widetilde\tau$ is faithful and acts on a separable Hilbert space.

Moreover, $A$ contains a unital copy $j:C\hookrightarrow A$ of the CAR algebra such that its normalized trace $\tau$ has a unique extension among all states of $A$, namely $\widetilde\tau$.
\end{theorem}

The algebra in \cref{thm:main} is therefore separably representable:
\[
 A\hookrightarrow\B(\ell^2(\N)).
\]
Its unique irreducible class, however, is realized only on nonseparable
Hilbert spaces. In fact, the faithful tracial representation has weak
closure the hyperfinite $\mathrm{II}_1$ factor; see \cref{cor:weak-closure}. It is not isomorphic to $\K(K)$ for any Hilbert
space $K$, since a unital algebra of compact operators is
finite-dimensional. Thus the theorem gives a negative answer to
Naimark's question without an additional set-theoretic hypothesis.

\subsection*{Relation to earlier constructions}
Glimm's theorem implies that a simple, separable $C^*$-algebra is either an
algebra of compact operators or has continuum many inequivalent
irreducible representations \cite{Glimm1961}; see also
\cite[\S5.5]{Farah2019}. Akemann and Weaver constructed a counterexample
to Naimark's problem under Jensen's principle $\diamondsuit_{\aleph_1}$
\cite{AW2004}. Their Proposition~6 states that a counterexample cannot be
generated by fewer than $\mathfrak c=2^{\aleph_0}$ elements. Its norm
density is therefore at least $\mathfrak c$, since any norm-dense subset
generates the algebra. Their Corollary~7 proves that the existence of an
$\aleph_1$-generated counterexample is independent of ZFC. This cardinality
restriction is essential: the result does not establish independence of
the existence of an unrestricted counterexample.

Calder\'on and Farah obtained counterexamples from
$\diamondsuit^{\mathrm{Cohen}}+\mathrm{CH}$, a hypothesis consistent with the
failure of Jensen's full diamond principle
\cite[Theorem~5.4 and Appendix~B]{CF2023}. They left open the existence of
a counterexample in ZFC alone, and also under CH alone \cite[\S8]{CF2023}.
The construction in \cref{thm:main} requires neither additional hypothesis.

Under diamond, Farah and Hirshberg constructed simple nuclear algebras
with prescribed finite or countably infinite numbers of irreducible
equivalence classes, together with further rigidity properties
\cite{FH2017}. Vaccaro showed that the tracial state space of a
counterexample can be any metrizable Choquet simplex, again under diamond
\cite{Vaccaro2018,Vaccaro2020}. In particular, the existence of a counterexample with a unique trace
under additional axioms is already known.
The tracial assertion of \cref{thm:main} combines unconditional existence
with a separable faithful GNS representation and uniqueness of the extension
of the CAR trace \emph{among all states}, rather than only among traces.

The distinction between faithful and irreducible representations is equally
important. Calder\'on and Farah's Theorem~B gives, under
$\diamondsuit^{\mathrm{Cohen}}+\mathrm{CH}$, separably represented simple
unital algebras with exactly $m$ irreducible equivalence classes for each
finite $m\ge2$ \cite{CF2023}. The present construction has one class.
Its separable tracial representation does not realize that class:
\cref{cor:no-separable-irreducible} shows that no nonzero irreducible
representation of $A$ is separable. In fact, its faithful tracial
representation contains no nonzero irreducible closed subrepresentation.

\subsection*{A clarification of separability}
In the singleton-spectrum setting, Rosenberg's theorem requires a
\emph{separable irreducible} representation, not merely a faithful
separable one. Some summaries use broader wording: Akemann--Weaver write
``no separable (indeed, no separably acting) counterexamples''
\cite[p.~7522]{AW2004}, Farah writes ``faithfully represented''
\cite[p.~155]{Farah2019}, and the abstract of Calder\'on--Farah refers to
``a separably represented, simple $C^*$-algebra'' \cite{CF2023}.
Read as excluding every faithfully separably represented counterexample,
these formulations are incompatible with \cref{thm:main}.
The parenthetical ``in fact, separably representable'' in the author's
earlier preprint \cite[p.~1]{Tanaka2025} is likewise too broad: it should
refer to the existence of an \emph{irreducible} representation on a
separable Hilbert space.

The precise later results retain irreducibility hypotheses:
\cite[Corollary~5.5.6, p.~158]{Farah2019} assumes an irreducible
representation on a Hilbert space of density less than $\mathfrak c$,
and \cite[Proposition~7.2]{CF2023} assumes that all irreducible
representations act on separable Hilbert spaces. Our algebra satisfies
neither hypothesis; its faithful separable tracial representation cannot
be substituted for an irreducible representation in these theorems.

The two representation sizes are determined exactly in \cref{cor:cardinal}:
\[
 |A|=\operatorname{dens}(A)=\operatorname{dens}(H)=\mathfrak c,
 \qquad H_\tau\text{ is separable and infinite-dimensional}.
\]
Here $H$ carries the unique irreducible class and $H_\tau$ the faithful
tracial representation. Combined with the Akemann--Weaver lower bound,
this gives the equivalence over ZFC between CH and the existence of a
counterexample of norm density $\aleph_1$ (\cref{cor:ch}). Neither CH
nor a prediction principle is used in the construction.

\subsection*{Outline of the proof}
The principal difficulty is to control representations of the final algebra.
Making selected pure states equivalent in a larger algebra does not suffice:
the larger algebra has its own pure states, which must also be treated.
This issue is explicit in Akemann--Weaver's construction and in Weaver's
account \cite{AW2004,Weaver2007}. Our approach is to fix a concrete
algebra and prove, from finite relations among its generators, a statement
about an arbitrary irreducible representation of that algebra.

The starting algebra is the CAR algebra $C$, with its normalized trace
$\tau$. A product pure state $\phi_{\rt}$ is excised by a decreasing sequence
of projections $q_n$ satisfying
\[
 \phi_{\rt}(q_n)=1,\qquad \tau(q_n)=2^{-n}.
\]
Choose pure states $\phi_i$, one in each irreducible GNS class of $C$.
The pure-state homogeneity theorem of Kishimoto--Ozawa--Sakai
\cite{KOS} provides approximately inner automorphisms transporting the
product-state data to each $\phi_i$. The resulting projections $p_{i,n}$
satisfy the corresponding compression identities, while retaining trace
$2^{-n}$. General pure-state excision is due to
Akemann--Anderson--Pedersen \cite{AAP1986}; for these particular CAR
projections we use a direct matrix calculation.

Differences of consecutive projections, called shells, are connected by
partial isometries $w_{i,n}\in C$. In the reduced atomic representation
$\Pi=\bigoplus_i\pi_i$, the strong sum of these partial isometries has
one-dimensional initial and final defects. A rank-one operator between
the corresponding pure-state vectors completes it to a unitary $L_i$.
Set
\[
 A=C^*(\Pi(C),\{L_i:i\in I\}),\qquad j(a)=\Pi(a)\quad(a\in C).
\]
The strong sums and rank-one defects are not separately adjoined. For
$s_{i,N}=\sum_{n=0}^{N-1}w_{i,n}\in C$, the construction gives the finite
identities
\[
 L_i(1-j(p_{i,N}))=j(s_{i,N}),\qquad
 L_i^*(1-j(q_N))=j(s_{i,N})^*.
\]
These identities can be transported to every representation of $A$;
they are the common input to both parts of the proof. Strong limits are
then formed in that representation, not transported from the atomic model.

The essential step is \cref{thm:capture}. For an arbitrary irreducible
representation $\rho$ of $A$, the finite identities determine $\rho(L_i)$
off the common ranges of the represented projection sequences. If all
common ranges vanished, $\rho|_{j(C)}$ would itself be irreducible,
contradicting the common-range calculation for the chosen GNS classes.
A unit vector in a nonzero common range therefore exists. Its translates
under the added unitaries generate mutually orthogonal cyclic
$C$-subspaces, one for each GNS class. The associated GNS intertwiners
assemble into an isometry from the atomic model. We compute the shell
and defect actions on its range and show that this isometry intertwines
both the generators and their adjoints. Its nonzero range therefore
reduces $\rho(A)$; irreducibility makes the isometry surjective.

The same projection sequences have a different effect on the trace.
Begin with an \emph{arbitrary state extension} $\Phi$ of $\tau$ to $A$.
On the cyclic subspace generated by $C$, the estimate
\[
 \|\rho_\Phi(j(p_{i,n}))\rho_\Phi(j(b))\Omega_\Phi\|^2
 \le \|b\|^2\tau(p_{i,n})=2^{-n}\|b\|^2
\]
forces every common fixed-space projection to vanish. The shell relations
then show that this cyclic subspace reduces all the generators of $A$.
It is therefore the full GNS space of $\Phi$, and is separable.
GNS uniqueness determines the images of the additional generators from
those of $C$, proving uniqueness of the state extension.
Only after this step do we prove traciality, by showing that the centralizer
contains the generators. Simplicity supplies faithfulness.

\subsection*{Relation to the preliminary report}
The Brief Report \emph{Naimark's problem} \cite{NaimarkBrief2026} is an
AI-assisted preliminary account of the same counterexample, its
irreducible representations, the extension of the CAR trace, and the
density and CH consequences. The present article gives a self-contained
treatment of that research with reorganized and expanded proofs and a
fuller comparison with earlier work; it also records the weak-closure
consequence in \cref{cor:weak-closure}. The report is a non-peer-reviewed
account, cited to document the provenance of the results and proof
strategy, not as an independent verification.

\subsection*{Organization}
Section~\ref{sec:flags} constructs the pure-state projections and their
shells. Section~\ref{sec:construction} defines the additional unitaries and
the algebra $A$, and Section~\ref{sec:capture} proves the classification of
its irreducible representations. Section~\ref{sec:tracial} establishes
unique extension of the CAR trace and completes the proof of the main
theorem. Section~\ref{sec:consequences} gives the representation-size and
cardinality consequences.

\subsection*{Conventions}
Throughout, algebras and Hilbert spaces are complex, and Hilbert-space
inner products are linear in the first variable. For vectors $x,y$, put
\[
 \theta_{x,y}z=\ip{z}{y}x,
 \qquad \theta_{x,y}^*=\theta_{y,x}.
\]
Thus $\theta_{x,x}$ is the orthogonal projection onto $\C x$ when
$\norm{x}=1$, and the vector state of a unit vector $\xi$ in a
representation $\pi$ is $a\mapsto\ip{\pi(a)\xi}{\xi}$.
A representation is a $*$-homomorphism; irreducible representations are
understood to be nonzero, with no separability assumption on their Hilbert
spaces. An irreducible representation of a unital algebra is automatically
unital. We write $\stlim$ for a strong operator limit. GNS uniqueness and
Schur's lemma are used in their usual Hilbert-space forms
\cite[\S\S1.10, 3.4--3.6]{Farah2019}. All section and result numbers in
\cite{Farah2019} refer to the first edition. The distinguished index
$\rt$ will denote the root state; projection sequences are indexed by
$n=0,1,\ldots$.

\section{Pure-state projections in the CAR algebra}\label{sec:flags}

We construct decreasing projections that excise each chosen pure state,
while their traces tend to zero. Their differences will supply the shell
links used in Section~\ref{sec:construction}.

\subsection{The product-state projections}
Let
\[
 C=\overline{\bigcup_{n\ge0}C_n},\qquad C_n=M_2(\C)^{\otimes n},
 \qquad C_n\longrightarrow C_{n+1},\quad a\longmapsto a\otimes1_2.
\]
The CAR algebra $C$ is unital, simple, separable, and infinite-dimensional.
Its unique tracial state is
\[
 \tau(a)=2^{-n}\operatorname{Tr}(a)\qquad(a\in C_n);
\]
see \cite[\S2.2 and Corollary~4.1.9]{Farah2019}. Write $e_{00}$ for the
first diagonal matrix unit of $M_2(\C)$, let $\phi_{\rt}$ be the product
state selecting the first coordinate in each tensor factor, and set
$q_n=e_{00}^{\otimes n}$. Then
\begin{equation}\label{eq:root-data}
 q_0=1,\qquad q_{n+1}\le q_n,\qquad
 \phi_{\rt}(q_n)=1,\qquad \tau(q_n)=2^{-n}.
\end{equation}
The state $\phi_{\rt}$ is pure: its restrictions to the matrix stages are
pure, and these stages have dense union
\cite[Lemma~5.5.3]{Farah2019}.

\begin{lemma}[Product-state compression]\label{lem:compression}
For every $a\in C$,
\begin{equation}\label{eq:root-compression}
 \norm{q_na q_n-\phi_{\rt}(a)q_n}\longrightarrow0
 \qquad(n\to\infty).
\end{equation}
\end{lemma}
\begin{proof}
For $b\in C_m$ and $n\ge m$, tensor-product compression gives
\[
 q_n(b\otimes1_{2^{n-m}})q_n=\phi_{\rt}(b)q_n.
\]
Consequently, for arbitrary $a\in C$,
\[
 \norm{q_na q_n-\phi_{\rt}(a)q_n}
 \le \norm{q_n(a-b)q_n}+|\phi_{\rt}(a-b)|
 \le2\norm{a-b}\qquad(n\ge m).
\]
Approximation by a matrix-stage element $b$ proves the assertion.
\end{proof}
This is the excision property of a pure state
\cite[Proposition~2.2]{AAP1986}, here realized by projections with the
explicit traces in \eqref{eq:root-data}.

\subsection{Transport to the other pure states}
Choose pure states $(\phi_i)_{i\in I}$ of $C$, with GNS triples
$(H_i,\pi_i,\xi_i)$, so that $(\pi_i)_{i\in I}$ contains exactly one
representative of each irreducible unitary equivalence class, with
$\rt\in I$ corresponding
to $\phi_{\rt}$. Such a family is obtained by choosing representatives
from the pure state space: every unit vector in an irreducible
representation is cyclic and has a pure vector state
\cite[Lemma~1.10.9 and Proposition~3.6.5]{Farah2019}.

We use the following form of the Kishimoto--Ozawa--Sakai homogeneity
theorem \cite[Theorems~1.1 and~2.1]{KOS}: pure states with the same GNS
kernel on a separable $C^*$-algebra are related by an asymptotically inner
automorphism whose implementing path starts at $1$. In our unital algebra
$C$, simplicity gives $\ker\pi_i=\ker\pi_{\rt}=0$. Thus we may choose
\begin{equation}\label{eq:alpha}
 \alpha_i\in\operatorname{AInn}_0(C),\qquad
 \phi_i\circ\alpha_i=\phi_{\rt},\qquad \alpha_{\rt}=\id_C.
\end{equation}
Here the notation means that there is a norm-continuous path
$t\mapsto u_{i,t}\in\U(C)$, $t\ge0$, such that
\begin{equation}\label{eq:asymptotic-inner}
 u_{i,0}=1,\qquad
 \lim_{t\to\infty}\norm{\alpha_i(a)-u_{i,t}au_{i,t}^*}=0
 \quad(a\in C).
\end{equation}
Only the following finite-set consequence is used to construct the links:
for every finite $F\subset C$ and every $\varepsilon>0$,
\begin{equation}\label{eq:approximate-inner}
 \text{there is }u\in\U(C)\text{ such that }
 \norm{\alpha_i(a)-uau^*}<\varepsilon\quad(a\in F).
\end{equation}
In particular, these implementing unitaries belong to the original
algebra $C$.

Define the transported projections and their shells by
\begin{equation}\label{eq:flags}
 p_{i,n}=\alpha_i(q_n),\qquad
 f_n=q_n-q_{n+1},\qquad
 e_{i,n}=p_{i,n}-p_{i,n+1}=\alpha_i(f_n).
\end{equation}
For fixed $i$, the $p_{i,n}$ decrease from $1$, and the $e_{i,n}$ are
pairwise orthogonal. We call $(p_{i,n})_{n\ge0}$ the $i$-flag. Applying
\cref{lem:compression} to $\alpha_i^{-1}(a)$ gives
\begin{align}
 \phi_i(p_{i,n})&=1,\notag\\
 \norm{p_{i,n}ap_{i,n}-\phi_i(a)p_{i,n}}
 &=\norm{q_n\alpha_i^{-1}(a)q_n
       -\phi_{\rt}(\alpha_i^{-1}(a))q_n}
 \longrightarrow0.\label{eq:transported-compression}
\end{align}
GNS uniqueness also gives a pointed unitary $Y_i:H_{\rt}\to H_i$ with
\begin{equation}\label{eq:twisted-gns}
 Y_i\xi_{\rt}=\xi_i,\qquad
 Y_i\pi_{\rt}(a)Y_i^*=\pi_i(\alpha_i(a))\quad(a\in C).
\end{equation}
This identifies $\pi_{\rt}$ with the \emph{twisted} representation
$\pi_i\circ\alpha_i$, not with $\pi_i$ itself.

\subsection{Shell links}
The partial isometries below send each transported shell to the root shell.

\begin{lemma}[Links and their traces]\label{lem:links}
There are elements $w_{i,n}\in C$ such that
\begin{equation}\label{eq:supports}
 w_{i,n}^*w_{i,n}=e_{i,n},\qquad
 w_{i,n}w_{i,n}^*=f_n,\qquad w_{\rt,n}=f_n.
\end{equation}
For these same choices,
\begin{equation}\label{eq:flag-trace}
 \tau(p_{i,n})=2^{-n}\qquad(i\in I,\ n\ge0).
\end{equation}
\end{lemma}
\begin{proof}
Fix $i\ne\rt$ and $n$. By \eqref{eq:approximate-inner}, choose
$u\in\U(C)$ with $\norm{e_{i,n}-uf_nu^*}<1$. The projection
perturbation lemma \cite[Lemma~1.5.7(2)]{Farah2019} gives $v\in\U(C)$
satisfying $ve_{i,n}v^*=uf_nu^*$. Set $w_{i,n}=u^*ve_{i,n}$. Then
\[
 \begin{aligned}
 w_{i,n}^*w_{i,n}
 &=e_{i,n}v^*uu^*ve_{i,n}=e_{i,n},\\
 w_{i,n}w_{i,n}^*
 &=u^*ve_{i,n}v^*u=f_n.
 \end{aligned}
\]
For $i=\rt$ take $w_{\rt,n}=f_n$. Since $\tau$ is tracial,
\[
 \tau(e_{i,n})=\tau(w_{i,n}^*w_{i,n})
 =\tau(w_{i,n}w_{i,n}^*)=\tau(f_n).
\]
Telescoping gives
\[
 \tau(p_{i,N})
 =1-\sum_{n=0}^{N-1}\tau(e_{i,n})
 =1-\sum_{n=0}^{N-1}\tau(q_n-q_{n+1})
 =\tau(q_N)=2^{-N}.
\]
\end{proof}
For later use, put $s_{i,N}=\sum_{n=0}^{N-1}w_{i,n}\in C$. Orthogonality
of the two families of supports gives
\begin{equation}\label{eq:finite-car-sums}
 s_{i,N}^*s_{i,N}=1-p_{i,N},\qquad
 s_{i,N}s_{i,N}^*=1-q_N,\qquad \norm{s_{i,N}}\le1.
\end{equation}

\subsection{Common ranges in a representation}
The lower index of a common-range projection will specify the flag;
its superscript will specify the representation. In particular, set
\[
 P_i^{(k)}=\stlim_{n\to\infty}\pi_k(p_{i,n})\in\B(H_k).
\]
Thus $P_i^{(k)}$ is the common-range projection of the $i$-flag in the
$k$th GNS representation.

\begin{lemma}[Common-range compression]\label{lem:fixed-space}
Let $\sigma:C\to\B(K)$ be a unital representation, and put
\[
 P_i^\sigma=\stlim_{n\to\infty}\sigma(p_{i,n}),\qquad
 \ran P_i^\sigma=\bigcap_{n\ge0}\sigma(p_{i,n})K.
\]
Then
\begin{equation}\label{eq:fixed-compression}
 P_i^\sigma\sigma(a)P_i^\sigma
 =\phi_i(a)P_i^\sigma\qquad(a\in C).
\end{equation}
For the chosen irreducible representations,
\begin{equation}\label{eq:gns-common-ranges}
 P_i^{(k)}=
 \begin{cases}
  \theta_{\xi_k,\xi_k},&i=k,\\
  0,&i\ne k.
 \end{cases}
\end{equation}
More generally, $P_i^\sigma=0$ if $\sigma$ is irreducible and inequivalent
to $\pi_i$.
\end{lemma}
\begin{proof}
The strong limit is the common-range projection
\cite[Lemma~3.1.3]{Farah2019}. By \eqref{eq:transported-compression},
\[
 \norm{\sigma(p_{i,n})\sigma(a)\sigma(p_{i,n})
       -\phi_i(a)\sigma(p_{i,n})}\longrightarrow0.
\]
The bounded products on the left converge strongly to the corresponding
products with $P_i^\sigma$, which proves \eqref{eq:fixed-compression}.

For the root GNS representation, write $P=P_{\rt}^{(\rt)}$. Since
\[
 \norm{(1-\pi_{\rt}(q_n))\xi_{\rt}}^2
 =1-\phi_{\rt}(q_n)=0,
\]
we have $P\xi_{\rt}=\xi_{\rt}$. The compression identity now gives
\[
 P\pi_{\rt}(a)\xi_{\rt}=\phi_{\rt}(a)\xi_{\rt}\qquad(a\in C).
\]
Since $\pi_{\rt}(C)\xi_{\rt}$ is dense, continuity of $P$ gives
$\ran P\subseteq\C\xi_{\rt}$. Together with $P\xi_{\rt}=\xi_{\rt}$,
this proves $P=\theta_{\xi_{\rt},\xi_{\rt}}$.
For arbitrary $k$, recall from \eqref{eq:twisted-gns} that
$Y_k\pi_{\rt}(a)Y_k^*=\pi_k(\alpha_k(a))$ and
$Y_k\xi_{\rt}=\xi_k$. Therefore
\begin{align*}
 P_k^{(k)}
 &=\stlim_{n\to\infty}\pi_k(\alpha_k(q_n))\\
 &=Y_k\left(\stlim_{n\to\infty}\pi_{\rt}(q_n)\right)Y_k^*\\
 &=Y_k\theta_{\xi_{\rt},\xi_{\rt}}Y_k^*
 =\theta_{\xi_k,\xi_k}.
\end{align*}
Finally, if $\sigma$ is irreducible and $\eta\in\ran P_i^\sigma$ is a
unit vector, then
\[
 \ip{\sigma(a)\eta}{\eta}
 =\ip{P_i^\sigma\sigma(a)P_i^\sigma\eta}{\eta}
 =\phi_i(a)\qquad(a\in C).
\]
Every nonzero vector of an irreducible representation is cyclic.
Thus GNS uniqueness identifies $\sigma$ with
$\pi_i$. This proves the vanishing assertion and the remaining case of
\eqref{eq:gns-common-ranges}.
\end{proof}

\section{The generated algebra}\label{sec:construction}

\subsection{Shell sums and defect spaces}
Matching the shells determines an operator only on the orthogonal
complement of the common range. The next lemma first assembles the shell
maps and then specifies the remaining map between the two common ranges.
Its converse will construct $L_i$; its forward direction will recover
$\rho(L_i)$ from the finite relations in any representation $\rho$.

\begin{lemma}[Shell and defect decomposition]\label{lem:shell-calculus}
Let $(P_n)$ and $(Q_n)$ be decreasing sequences of orthogonal projections
on a Hilbert space $K$, with $P_0=Q_0=1$. Suppose
\[
 W_n^*W_n=P_n-P_{n+1},\qquad
 W_nW_n^*=Q_n-Q_{n+1}\qquad(n\ge0).
\]
Write $P=\stlim_{n\to\infty}P_n$ and $Q=\stlim_{n\to\infty}Q_n$.
The sums $S_N=\sum_{n=0}^{N-1}W_n$ and $S_N^*$ converge strongly to
$S$ and $S^*$, respectively, with
\begin{equation}\label{eq:shell-products}
 S^*S=1-P,\qquad SS^*=1-Q.
\end{equation}
If $V$ is a unitary satisfying $V(P_n-P_{n+1})=W_n$, then
\begin{equation}\label{eq:reconstruction}
 V=S+D,\qquad D=VP=QDP,\qquad D^*D=P,\quad DD^*=Q.
\end{equation}
Conversely, for any $D$ with $D^*D=P$ and $DD^*=Q$, the operator
$V=S+D$ is unitary and satisfies
\begin{equation}\label{eq:shell-unitary-action}
 V(P_n-P_{n+1})=W_n,\qquad VP_nV^*=Q_n\qquad(n\ge0).
\end{equation}
It is the unique such unitary with $V|_{PK}=D|_{PK}$.
In particular, when $P=\theta_{\xi,\xi}$ and $Q=\theta_{\eta,\eta}$
for unit vectors $\xi,\eta$, the unique shell-matching unitary sending
$\xi$ to $\eta$ is $S+\theta_{\eta,\xi}$.
\end{lemma}
\begin{proof}
Set $F_n=P_n-P_{n+1}$ and $G_n=Q_n-Q_{n+1}$. The support identities
$W_n=G_nW_nF_n$ and orthogonality give, for $m\ne n$,
\[
 W_m^*W_n=W_m^*G_mG_nW_n=0,\qquad
 W_mW_n^*=W_mF_mF_nW_n^*=0.
\]
Consequently
\begin{equation}\label{eq:finite-shell-products}
 S_N^*S_N=1-P_N,\qquad S_NS_N^*=1-Q_N.
\end{equation}
For $M\ge N$, the decreasing projections satisfy
\[
 0\le P_N-P_M\le P_N-P,\qquad
 0\le Q_N-Q_M\le Q_N-Q.
\]
Since $P_N-P$ and $Q_N-Q$ are themselves projections, for $x\in K$ we have
\[
 \begin{aligned}
 \norm{(S_M-S_N)x}^2
 &=\ip{(P_N-P_M)x}{x}\\
 &\le\ip{(P_N-P)x}{x}=\norm{(P_N-P)x}^2,\\
 \norm{(S_M^*-S_N^*)x}^2
 &=\ip{(Q_N-Q_M)x}{x}\\
 &\le\ip{(Q_N-Q)x}{x}=\norm{(Q_N-Q)x}^2.
 \end{aligned}
\]
Both bounds tend to zero as $N\to\infty$, uniformly for $M\ge N$.
Thus the sums and their adjoints converge strongly. Since $\norm{S_N}\le1$ and
$\ip{S_Nx}{y}=\ip{x}{S_N^*y}$, their limits are adjoints. Taking strong
limits in \eqref{eq:finite-shell-products} proves \eqref{eq:shell-products}.
In particular,
\[
 S=(1-Q)S(1-P).
\]

If $V$ satisfies the finite shell relations, then
\begin{equation}\label{eq:finite-reconstruction}
 V(1-P_N)=S_N,\qquad
 V(1-P_N)V^*=1-Q_N,\qquad V^*(1-Q_N)=S_N^*.
\end{equation}
Taking strong limits gives $V(1-P)=S$ and $VPV^*=Q$.
For $D=VP$ we therefore have
\[
 D=QDP,\qquad D^*D=PV^*VP=P,\qquad DD^*=VPV^*=Q.
\]
This proves \eqref{eq:reconstruction}.

Conversely, $D^*D=P$ and $DD^*=Q$ imply $D=QDP$, so
\[
 S^*D=D^*S=0,\qquad SD^*=DS^*=0.
\]
Hence, for $V=S+D$,
\[
 V^*V=(1-P)+P=1,\qquad VV^*=(1-Q)+Q=1.
\]
Also $SF_n=W_n$ and $DF_n=0$, whence $VF_n=W_n$.
Summing these identities for $n<N$ proves \eqref{eq:finite-reconstruction}
and thus $VP_NV^*=Q_N$. If $\widetilde V$ has the same shell and defect
actions, then
\[
 (\widetilde V-V)(1-P_N)=0,\qquad (\widetilde V-V)P=0.
\]
Letting $N\to\infty$ strongly gives $\widetilde V=V$.
For rank-one $P,Q$, take $D=\theta_{\eta,\xi}$.
\end{proof}

\subsection{The atomic model}
Form the reduced atomic representation
\[
 H=\bigoplus_{i\in I}H_i,\qquad
 \Pi=\bigoplus_{i\in I}\pi_i,\qquad
 \zeta_i=\iota_i\xi_i,
\]
where $\iota_i:H_i\to H$ is the coordinate inclusion. Each $\pi_i$ is
faithful by simplicity of $C$, so $\Pi$ is faithful. The common-range
calculation \eqref{eq:gns-common-ranges} gives
\begin{equation}\label{eq:atomic-limits}
 P_i^\Pi:=\stlim_{n\to\infty}\Pi(p_{i,n})
 =\theta_{\zeta_i,\zeta_i},\qquad
 P_{\rt}^\Pi=\stlim_{n\to\infty}\Pi(q_n)
 =\theta_{\zeta_{\rt},\zeta_{\rt}}.
\end{equation}
The limits hold first on finite-support vectors and then on $H$, since the
projections are uniformly bounded.

Recall the finite sums $s_{i,N}=\sum_{n=0}^{N-1}w_{i,n}$ in $C$.
Define the shell part, the model defect part, and their sum by
\begin{equation}\label{eq:generators}
 S_i=\stlim_{N\to\infty}\Pi(s_{i,N}),\qquad
 R_i^{\mathrm{def}}=\theta_{\zeta_{\rt},\zeta_i},\qquad
 L_i=S_i+R_i^{\mathrm{def}}.
\end{equation}
Their initial and final projections are complementary:
\[
 \begin{aligned}
 S_i^*S_i&=1-P_i^\Pi,&
 S_iS_i^*&=1-P_{\rt}^\Pi,\\
 (R_i^{\mathrm{def}})^*R_i^{\mathrm{def}}&=P_i^\Pi,&
 R_i^{\mathrm{def}}(R_i^{\mathrm{def}})^*&=P_{\rt}^\Pi.
 \end{aligned}
\]
The mixed products vanish. By \cref{lem:shell-calculus},
\begin{equation}\label{eq:generator-relations}
 \begin{gathered}
 L_i\in\U(\B(H)),\qquad L_i\zeta_i=\zeta_{\rt},\qquad L_{\rt}=1,\\
 L_i\Pi(e_{i,n})=\Pi(w_{i,n}),\qquad
 L_i\Pi(p_{i,n})L_i^*=\Pi(q_n).
 \end{gathered}
\end{equation}
At the root, $S_{\rt}=1-P_{\rt}^\Pi$ and
$R_{\rt}^{\mathrm{def}}=P_{\rt}^\Pi$, which explains $L_{\rt}=1$.

Set
\begin{equation}\label{eq:algebra}
 A=C^*(\Pi(C),\{L_i:i\in I\})\subseteq\B(H),\qquad
 j(a)=\Pi(a)\quad(a\in C).
\end{equation}
Only $\Pi(C)$ and the $L_i$ are adjoined. The strong limits $P_i^\Pi$ and
$S_i$, and the rank-one operators $R_i^{\mathrm{def}}$, are used in
$\B(H)$; their separate membership in $A$ is not required.

Summing the shell relations gives the following identities in $A$:
\begin{equation}\label{eq:finite-trace-relations}
 \begin{aligned}
 L_i(1-j(p_{i,N}))&=j(s_{i,N}),\\
 L_i^*(1-j(q_N))&=j(s_{i,N})^*.
 \end{aligned}
\end{equation}
In addition, $L_i j(p_{i,N})L_i^*=j(q_N)$ by
\eqref{eq:generator-relations}. We use these identities in
Sections~\ref{sec:capture} and~\ref{sec:tracial}. Being finite identities
in $A$, they are preserved by every $*$-homomorphism from $A$.

\begin{proposition}\label{prop:atomic}
The algebra $A$ is unital and infinite-dimensional, $j$ is injective,
and the inclusion $A\subseteq\B(H)$ is irreducible.
\end{proposition}
\begin{proof}
The first assertions follow from the faithful unital copy of $C$.
If $T\in\Pi(C)'$, its matrix blocks satisfy
\[
 (\iota_j^*T\iota_i)\pi_i(a)
 =\pi_j(a)(\iota_j^*T\iota_i)\qquad(a\in C).
\]
Schur's lemma therefore gives
\[
 \Pi(C)'=
 \left\{\bigoplus_{i\in I}\lambda_i1_{H_i}:
                  \sup_i|\lambda_i|<\infty\right\}.
\]
If $T$ also commutes with every $L_i$, then
\[
 \lambda_{\rt}\zeta_{\rt}
 =TL_i\zeta_i=L_iT\zeta_i=\lambda_i\zeta_{\rt}.
\]
Thus all $\lambda_i$ are equal, so $A'=\C1$ and the inclusion is
irreducible.
\end{proof}

\section{All irreducible representations}\label{sec:capture}

We now recover the atomic model inside an arbitrary irreducible
representation. The two parts of each added unitary must be controlled
separately: the shells determine its action off the common range, while
one vector in a nonzero common range determines the required defect action
on the cyclic subspaces constructed below.

\begin{theorem}\label{thm:capture}
Every nonzero irreducible representation $\rho:A\to\B(K)$ is unitarily
equivalent to the inclusion $A\subseteq\B(H)$.
\end{theorem}
\begin{proof}
The representation $\rho$ is unital. Write
\[
 \sigma=\rho\circ j:C\longrightarrow\B(K),\qquad V_i=\rho(L_i).
\]
We will construct an isometry $W:H\to K$, prove the generator relations
\[
 W\Pi(a)=\sigma(a)W,\qquad
 WL_i=V_iW,\qquad WL_i^*=V_i^*W,
\]
and then use irreducibility to show that $W$ is surjective.

\emph{Step 1: a nonzero common range.}
On $K$, put
\[
 P_i^\sigma=\stlim_{n\to\infty}\sigma(p_{i,n}),\qquad
 P_{\rt}^\sigma=\stlim_{n\to\infty}\sigma(q_n).
\]
Applying $\rho$ to the finite relations in
\eqref{eq:generator-relations} gives
$V_i\sigma(e_{i,n})=\sigma(w_{i,n})$. By
\cref{lem:shell-calculus},
\begin{align}
 S_i^\sigma&=\stlim_{N\to\infty}\sigma(s_{i,N}),
 & (S_i^\sigma)^*&=\stlim_{N\to\infty}\sigma(s_{i,N})^*,
 \label{eq:represented-sum}\\
 V_i&=S_i^\sigma+D_i,
 & D_i&=V_iP_i^\sigma=P_{\rt}^\sigma D_iP_i^\sigma,
 \label{eq:represented-decomposition}\\
 D_i^*D_i&=P_i^\sigma,
 & D_iD_i^*&=P_{\rt}^\sigma.
 \label{eq:represented-products}
\end{align}
Here $s_{i,N}=\sum_{n=0}^{N-1}w_{i,n}\in C$. These limits are formed
from the represented finite sums, not by applying $\rho$ to a strong
limit on $H$.

Suppose $P_i^\sigma=0$ for every $i$. Then $V_i=S_i^\sigma$. If
$F\in\sigma(C)'$ is an orthogonal projection, then
\[
 F\sigma(s_{i,N})=\sigma(s_{i,N})F
 \quad\Longrightarrow\quad FV_i=V_iF
\]
by strong convergence. Taking adjoints also gives $FV_i^*=V_i^*F$.
Thus $F\in\rho(A)'=\C1$, so $F$ is $0$ or $1$. Hence $\sigma$ has no
nontrivial reducing subspace and is irreducible. The family chosen in
Section~\ref{sec:flags} supplies
$k\in I$ and a unitary $Z:H_k\to K$ with
$\sigma(a)=Z\pi_k(a)Z^*$. But \eqref{eq:gns-common-ranges} gives
\[
 0=P_k^\sigma
 =Z\left(\stlim_{n\to\infty}\pi_k(p_{k,n})\right)Z^*
 =\theta_{Z\xi_k,Z\xi_k}\ne0,
\]
a contradiction. Some $P_i^\sigma$ is nonzero; then
\eqref{eq:represented-products} implies $P_{\rt}^\sigma\ne0$.
Moreover, $D_i=V_iP_i^\sigma$ and $D_iD_i^*=P_{\rt}^\sigma$ give
\begin{equation}\label{eq:all-common-ranges}
 V_iP_i^\sigma V_i^*=P_{\rt}^\sigma,\qquad
 P_i^\sigma=V_i^*P_{\rt}^\sigma V_i\ne0\qquad(i\in I).
\end{equation}
Thus every flag has a nonzero common range in this representation.

\emph{Step 2: orthogonal copies of the chosen GNS spaces.}
Choose a unit vector $\eta_{\rt}\in P_{\rt}^\sigma K$ and set
$\eta_i=V_i^*\eta_{\rt}$. By \eqref{eq:all-common-ranges},
\begin{equation}\label{eq:eta}
 \eta_i=V_i^*\eta_{\rt},\qquad
 P_i^\sigma\eta_i=\eta_i,\qquad
 \norm{\eta_i}=1,\qquad V_i\eta_i=\eta_{\rt}.
\end{equation}
At $i=\rt$ this agrees with the chosen vector, because $V_{\rt}=1$.
The compression identity \eqref{eq:fixed-compression} gives
\begin{equation}\label{eq:eta-state}
 \ip{\sigma(a)\eta_i}{\eta_i}
 =\ip{P_i^\sigma\sigma(a)P_i^\sigma\eta_i}{\eta_i}
 =\phi_i(a)\qquad(a\in C).
\end{equation}
Put
\[
 M_i=\overline{\sigma(C)\eta_i},\qquad
 \sigma_i(a)=\sigma(a)|_{M_i}.
\]
Let $E_i\in\B(K)$ be the orthogonal projection onto $M_i$. Each $M_i$
reduces $\sigma(C)$, and $\eta_i$ is cyclic in $M_i$ by definition.
GNS uniqueness gives a pointed unitary
$T_i:H_i\to M_i$ satisfying
\begin{equation}\label{eq:pointed-gns}
 T_i\pi_i(a)\xi_i=\sigma(a)\eta_i,\qquad
 T_i\xi_i=\eta_i,\qquad
 T_i\pi_i(a)=\sigma_i(a)T_i.
\end{equation}
In particular, the $\sigma_i$ are irreducible and pairwise inequivalent.
For $i\ne j$, the projection relation
$E_j\sigma(a)=\sigma(a)E_j$ shows that
\[
 Z_{ji}:=E_j|_{M_i}:M_i\to M_j,\qquad
 Z_{ji}\sigma_i(a)=\sigma_j(a)Z_{ji}\quad(a\in C).
\]
Schur's lemma gives $Z_{ji}=0$, whence
\begin{equation}\label{eq:cyclic-orthogonality}
 M_i\perp M_j\qquad(i\ne j).
\end{equation}
The $T_i$ consequently define an isometry $W:H\to K$ by
\[
 W(\iota_i x)=T_ix\quad(x\in H_i),\qquad
 \ran W=M:=\bigoplus_{i\in I}M_i.
\]
Its range is closed and reduces $\sigma(C)$, and
\begin{equation}\label{eq:source-intertwining}
 W\zeta_i=\eta_i,\qquad W\Pi(a)=\sigma(a)W\quad(a\in C).
\end{equation}
Surjectivity of $W$ has not yet been established.

\emph{Step 3: the defect action on $M$.}
We first restrict the common-range projections to each summand. For
$x\in H_k$, the relation $\sigma(a)T_k=T_k\pi_k(a)$ gives
\[
 \begin{aligned}
 P_i^\sigma T_kx
 &=\lim_{n\to\infty}\sigma(p_{i,n})T_kx\\
 &=T_k\left(\lim_{n\to\infty}\pi_k(p_{i,n})x\right)
 =T_kP_i^{(k)}x,
 \end{aligned}
\]
where the limits on vectors are in Hilbert-space norm. Also $E_k$
commutes with $\sigma(p_{i,n})$ and its strong limit, so $M_k$ reduces
$P_i^\sigma$. By \eqref{eq:gns-common-ranges}, as operators on $M_k$,
\begin{equation}\label{eq:range-on-summands}
 P_i^\sigma|_{M_k}
 =T_kP_i^{(k)}T_k^*
 =\delta_{ik}\theta_{\eta_k,\eta_k}.
\end{equation}
Here $\delta_{ik}$ is the Kronecker delta. Passing from finite sums of
summands to their closure gives
\begin{equation}\label{eq:range-compressions}
 P_i^\sigma|_M=\theta_{\eta_i,\eta_i},\qquad
 P_i^\sigma M=\C\eta_i,\qquad
 P_{\rt}^\sigma M=\C\eta_{\rt}.
\end{equation}
Now recall $D_i=V_iP_i^\sigma$ and
$D_i^*=P_i^\sigma D_i^*P_{\rt}^\sigma$. By \eqref{eq:eta},
\[
 D_i\eta_i=V_iP_i^\sigma\eta_i=\eta_{\rt},\qquad
 D_i^*\eta_{\rt}=P_i^\sigma V_i^*\eta_{\rt}=\eta_i.
\]
Consequently, for $x\in M$,
\begin{equation}\label{eq:defect-on-cyclic-sum}
 D_ix=\ip{x}{\eta_i}\eta_{\rt},\qquad
 D_i^*x=\ip{x}{\eta_{\rt}}\eta_i.
\end{equation}
These are rank-one formulas on $M$ only; no dimension assertion about
$P_i^\sigma K$ has been used.

\emph{Step 4: intertwining the generators and proving surjectivity.}
Recall the model decomposition
\[
 S_i=\stlim_{N\to\infty}\Pi(s_{i,N}),\qquad
 R_i^{\mathrm{def}}=\theta_{\zeta_{\rt},\zeta_i},\qquad
 L_i=S_i+R_i^{\mathrm{def}}.
\]
At each finite stage, \eqref{eq:source-intertwining} intertwines
$\Pi(s_{i,N})$ with $\sigma(s_{i,N})$ and also intertwines their
adjoints. Taking the two strong limits yields
\begin{equation}\label{eq:shell-intertwining}
 S_i^\sigma W=WS_i,\qquad (S_i^\sigma)^*W=WS_i^*.
\end{equation}
For the defect part, use \eqref{eq:defect-on-cyclic-sum} and
$W\zeta_i=\eta_i$. For $x\in H$, preservation of the inner product gives
\begin{align}
 D_iWx
 &=V_iP_i^\sigma Wx
 =\ip{Wx}{\eta_i}\eta_{\rt}\notag\\
 &=\ip{x}{\zeta_i}W\zeta_{\rt}
 =WR_i^{\mathrm{def}}x,\label{eq:defect-intertwining}\\
 D_i^*Wx
 &=\ip{Wx}{\eta_{\rt}}\eta_i
 =\ip{x}{\zeta_{\rt}}W\zeta_i
 =W(R_i^{\mathrm{def}})^*x.\label{eq:adjoint-defect-intertwining}
\end{align}
Thus both required generator relations hold:
\begin{equation}\label{eq:full-generator-intertwining}
 \begin{aligned}
 V_iW&=(S_i^\sigma+D_i)W=W(S_i+R_i^{\mathrm{def}})=WL_i,\\
 V_i^*W&=((S_i^\sigma)^*+D_i^*)W
             =W(S_i^*+(R_i^{\mathrm{def}})^*)=WL_i^*.
 \end{aligned}
\end{equation}
Together with \eqref{eq:source-intertwining}, these identities show that
$M=\ran W$ is invariant under $\sigma(C)$, every $V_i$, and every $V_i^*$.
Hence $M$ reduces $\rho(A)$. Since $0\ne\eta_{\rt}\in M$,
irreducibility gives $M=K$, and only now is $W$ unitary.
Finally, norm generation extends the intertwining identities to
$Wb=\rho(b)W$ for every $b\in A$.
\end{proof}

\begin{corollary}\label{cor:simple}
The algebra $A$ is simple and is not isomorphic to an algebra of compact
operators.
\end{corollary}
\begin{proof}
If $J\subset A$ were a nonzero proper closed ideal, an irreducible
representation of the nonzero quotient $A/J$ would pull back to an
irreducible representation $\rho$ of $A$ whose kernel contains $J$.
Such a representation exists by the pure-state GNS construction
\cite[\S3.6]{Farah2019}. By \cref{thm:capture}, $\rho$ is unitarily
equivalent to the concrete inclusion of $A$ and is therefore faithful.
Thus $\ker\rho=0$, contradicting $0\ne J\subseteq\ker\rho$.
Hence $A$ is simple. A compact-operator algebra
$\K(K)$ is unital only when $K$ is finite-dimensional, whereas $A$ is
unital and infinite-dimensional.
\end{proof}

\section{The tracial representation}\label{sec:tracial}

The same flags now have the opposite effect. In the irreducible model their
common ranges are rank one and the defect terms remain; in a state extension
of $\tau$, the bounds $\tau(p_{i,n})=\tau(q_n)=2^{-n}$ force the common
ranges to vanish. The generators will therefore be determined by the
shell sums alone.

\subsection{Vanishing of the common ranges}
By state extension \cite[Lemma~1.7.6(3)]{Farah2019}, choose an arbitrary
state $\Phi$ of $A$ satisfying $\Phi\circ j=\tau$. In its GNS triple,
put
\[
 (K_\Phi,\rho_\Phi,\Omega_\Phi),\qquad
 \sigma_\Phi=\rho_\Phi\circ j,\qquad
 M_\Phi=\overline{\sigma_\Phi(C)\Omega_\Phi}.
\]
At this point only $\tau$ on $C$, not $\Phi$ on $A$, is known to be
tracial. We first show $M_\Phi=K_\Phi$, then compare all extensions on
the fixed GNS space $(H_\tau,\pi_\tau,\Omega_\tau)$ of $\tau$. Uniqueness
will be proved before traciality.

\begin{lemma}[Trace estimate]\label{lem:trace-bound}
Suppose $\sigma:C\to\B(K)$ is unital and a unit vector $\Omega$ has
vector state $\tau$. For every projection $p\in C$ and every $b\in C$,
\begin{equation}\label{eq:trace-bound}
 \norm{\sigma(p)\sigma(b)\Omega}^2\le\norm{b}^2\tau(p).
\end{equation}
Thus $\sigma(p_{i,n})x\to0$ and $\sigma(q_n)x\to0$ in Hilbert-space
norm for every $x\in\overline{\sigma(C)\Omega}$.
\end{lemma}
\begin{proof}
The tracial $L^2$ calculation is
\[
 \begin{aligned}
 \norm{\sigma(p)\sigma(b)\Omega}^2
 &=\tau(b^*pb)\\
 &=\tau(pbb^*)\\
 &=\tau(pbb^*p)\le\norm{b}^2\tau(p).
 \end{aligned}
\]
Here the insertion of the final $p$ uses
\[
 \tau(pbb^*(1-p))=\tau((1-p)pbb^*)=0,
\]
and the inequality follows from $0\le pbb^*p\le\norm{b}^2p$.
Using $\tau(p_{i,n})=\tau(q_n)=2^{-n}$ gives
\[
 \norm{\sigma(p_{i,n})\sigma(b)\Omega}
 \le2^{-n/2}\norm{b},\qquad
 \norm{\sigma(q_n)\sigma(b)\Omega}
 \le2^{-n/2}\norm{b}.
\]
For $x\in\overline{\sigma(C)\Omega}$, contractivity of the projections
therefore gives
\[
 \norm{\sigma(p_{i,n})x}
 \le\norm{x-\sigma(b)\Omega}+2^{-n/2}\norm{b}.
\]
First approximate $x$ by $\sigma(b)\Omega$ and then let $n\to\infty$.
The same argument applies to $q_n$.
\end{proof}

\subsection{The full GNS space}
For $s_{i,N}=\sum_{n=0}^{N-1}w_{i,n}\in C$, recall
\eqref{eq:finite-trace-relations}:
\[
 L_i(1-j(p_{i,N}))=j(s_{i,N}),\qquad
 L_i^*(1-j(q_N))=j(s_{i,N})^*.
\]
Their two remainder terms vanish on $M_\Phi$ by the trace estimate.
This will show that adjoining the generators does not enlarge the
CAR-cyclic Hilbert space.

\begin{proposition}\label{prop:trace-cyclic}
For every state extension $\Phi$ of $\tau$ to $A$, one has
$M_\Phi=K_\Phi$. In this representation,
\begin{equation}\label{eq:trace-generators}
 \rho_\Phi(L_i)=\stlim_{N\to\infty}\sigma_\Phi(s_{i,N}),\qquad
 \rho_\Phi(L_i)^*=\stlim_{N\to\infty}\sigma_\Phi(s_{i,N})^*.
\end{equation}
In particular, $K_\Phi$ is separable.
\end{proposition}
\begin{proof}
Apply $\rho_\Phi$ to \eqref{eq:finite-trace-relations}. For $x\in M_\Phi$,
unitarity and \cref{lem:trace-bound} give
\begin{align}
 \norm{\rho_\Phi(L_i)x-\sigma_\Phi(s_{i,N})x}
 &=\norm{\sigma_\Phi(p_{i,N})x}\longrightarrow0,\label{eq:trace-remainder}\\
 \norm{\rho_\Phi(L_i)^*x-\sigma_\Phi(s_{i,N})^*x}
 &=\norm{\sigma_\Phi(q_N)x}\longrightarrow0.\label{eq:trace-adjoint-remainder}
\end{align}
The finite sums and their adjoints preserve $M_\Phi$, since this space
reduces $\sigma_\Phi(C)$. Thus
\[
 \rho_\Phi(L_i)M_\Phi\subseteq M_\Phi,\qquad
 \rho_\Phi(L_i)^*M_\Phi\subseteq M_\Phi\qquad(i\in I).
\]
It follows that $M_\Phi$ reduces $\rho_\Phi(A)$. As it contains
$\Omega_\Phi$,
\[
 K_\Phi=\overline{\rho_\Phi(A)\Omega_\Phi}\subseteq M_\Phi.
\]
This proves equality. The preceding vector-norm limits consequently hold
for every $x\in K_\Phi$, giving \eqref{eq:trace-generators}.
Finally, $C$ is separable and its orbit of $\Omega_\Phi$ is dense in
$K_\Phi$, so $K_\Phi$ is separable.
\end{proof}

\subsection{Uniqueness and traciality}
\begin{theorem}\label{thm:trace-extension}
The normalized trace $\tau$ of $C$ has a unique state extension
$\widetilde\tau$ to $A$. This extension is tracial and faithful, and it is
the only tracial state of $A$.
\end{theorem}
\begin{proof}
We first identify every extension on the fixed CAR GNS space. This determines
the state uniquely; only then do we prove traciality from its centralizer,
and faithfulness from simplicity.

\emph{Uniqueness of the state extension.}
Existence follows from state extension. Fix an extension $\Phi$.
By \cref{prop:trace-cyclic}, the rule
\[
 T_\Phi\pi_\tau(b)\Omega_\tau
 =\sigma_\Phi(b)\Omega_\Phi\qquad(b\in C)
\]
extends to a unitary $T_\Phi:H_\tau\to K_\Phi$: its range is dense, and
\[
 \ip{\sigma_\Phi(b)\Omega_\Phi}{\sigma_\Phi(c)\Omega_\Phi}
 =\tau(c^*b)
 =\ip{\pi_\tau(b)\Omega_\tau}{\pi_\tau(c)\Omega_\tau}.
\]
It satisfies $T_\Phi\Omega_\tau=\Omega_\Phi$ and
$T_\Phi\pi_\tau(b)=\sigma_\Phi(b)T_\Phi$. Conjugating the existing
representation $\rho_\Phi$, define
\begin{equation}\label{eq:theta}
 \Theta_\Phi(a)=T_\Phi^*\rho_\Phi(a)T_\Phi\qquad(a\in A).
\end{equation}
Thus $\Theta_\Phi$ is a representation of the concrete algebra $A$, not
merely an assignment of operators to its generators. By
\eqref{eq:trace-generators}, its generator images are
\begin{equation}\label{eq:theta-images}
 \Theta_\Phi(j(b))=\pi_\tau(b),\qquad
 \Theta_\Phi(L_i)=\stlim_{N\to\infty}\pi_\tau(s_{i,N}).
\end{equation}
Both right sides are independent of $\Phi$. If $\Psi$ is another
extension, norm generation therefore gives $\Theta_\Phi=\Theta_\Psi$.
Since both GNS identifications preserve the cyclic vector,
\[
 \Phi(a)=\ip{\Theta_\Phi(a)\Omega_\tau}{\Omega_\tau}
        =\ip{\Theta_\Psi(a)\Omega_\tau}{\Omega_\tau}
        =\Psi(a)\qquad(a\in A).
\]
Write $\widetilde\tau$ for this unique extension.

\emph{Traciality.}
For $u\in\U(C)$, the state
$a\mapsto\widetilde\tau(j(u)^*aj(u))$ again restricts to $\tau$.
Uniqueness therefore gives
\[
 \widetilde\tau(j(u)^*aj(u))=\widetilde\tau(a)\qquad(a\in A).
\]
Replacing $a$ by $j(u)a$ yields
$\widetilde\tau(aj(u))=\widetilde\tau(j(u)a)$. Since unitaries linearly
span $C$ \cite[Exercise~1.11.16(4)--(5)]{Farah2019},
\begin{equation}\label{eq:car-centralizer}
 \widetilde\tau(j(b)a)=\widetilde\tau(aj(b))
 \qquad(b\in C,\ a\in A).
\end{equation}

In the GNS triple $(K,\rho,\Omega)$ of $\widetilde\tau$, recall
$s_{i,N}=\sum_{n=0}^{N-1}w_{i,n}$ and
$\rho(j(s_{i,N}))\to\rho(L_i)$ strongly. For $a\in A$,
\begin{align*}
 \widetilde\tau(L_i a)
 &=\ip{\rho(L_i)\rho(a)\Omega}{\Omega}\\
 &=\lim_{N\to\infty}\widetilde\tau(j(s_{i,N})a)\\
 &=\lim_{N\to\infty}\widetilde\tau(aj(s_{i,N}))
 =\ip{\rho(a)\rho(L_i)\Omega}{\Omega}
 =\widetilde\tau(aL_i).
\end{align*}
The limits use strong convergence on the two vectors $\rho(a)\Omega$
and $\Omega$, and the middle equality is \eqref{eq:car-centralizer}.
Thus the centralizer
\[
 Z_{\widetilde\tau}
 =\{x\in A:\widetilde\tau(xa)=\widetilde\tau(ax)
                    \text{ for every }a\in A\}
\]
contains $j(C)$ and all $L_i$. It is a norm-closed linear subspace containing
$1$, by continuity. For $x,y\in Z_{\widetilde\tau}$ and $a\in A$,
\[
 \begin{aligned}
 \widetilde\tau(xya)&=\widetilde\tau(yax)=\widetilde\tau(axy),\\
 \widetilde\tau(x^*a)
 &=\overline{\widetilde\tau(a^*x)}
  =\overline{\widetilde\tau(xa^*)}=\widetilde\tau(ax^*).
 \end{aligned}
\]
Hence it is a $C^*$-subalgebra. Since it contains every generator,
$Z_{\widetilde\tau}=A$, proving traciality.

\emph{Uniqueness of the trace and faithfulness.}
Every tracial state of $A$ restricts to the unique trace $\tau$ of $C$,
and hence equals $\widetilde\tau$. Since $A$ is simple by
\cref{cor:simple}, its tracial state is faithful
\cite[Corollary~4.1.4]{Farah2019}: $\widetilde\tau(a^*a)=0$ implies
$a=0$. The representation $\Theta_{\widetilde\tau}$ on $H_\tau$ is the
tracial GNS representation in these coordinates, and is faithful as well.
\end{proof}

\begin{corollary}[The tracial weak closure]\label{cor:weak-closure}
Put $R=\pi_\tau(C)''\subseteq\B(H_\tau)$. Then $R$ is a hyperfinite
$\mathrm{II}_1$ factor, and the faithful representation
$\Theta=\Theta_{\widetilde\tau}$ from \eqref{eq:theta} satisfies
\begin{equation}\label{eq:tracial-weak-closure}
 \pi_\tau(C)\subseteq\Theta(A)\subseteq R,\qquad \Theta(A)''=R.
\end{equation}
Thus $A$ is isomorphic to a strongly dense unital $C^*$-subalgebra of $R$.
\end{corollary}
\begin{proof}
By \eqref{eq:theta-images},
\[
 \Theta(j(b))=\pi_\tau(b)\quad(b\in C),\qquad
 \Theta(L_i)=\stlim_{N\to\infty}\pi_\tau(s_{i,N})\in R.
\]
Since $R$ is strongly closed and hence norm closed, norm generation of $A$
gives both inclusions in \eqref{eq:tracial-weak-closure}. Taking
bicommutants gives equality; the bicommutant theorem gives strong density.

For completeness, the matrix algebras $\pi_\tau(C_n)$ have strongly dense
union, so $R$ is hyperfinite and infinite-dimensional. Define the normal
vector state
\[
 \widehat\tau(x)=\ip{x\Omega_\tau}{\Omega_\tau}\qquad(x\in R).
\]
For each $a\in\pi_\tau(C)$, the normal functionals
\[
 x\longmapsto\widehat\tau(ax),\qquad
 x\longmapsto\widehat\tau(xa)
\]
agree on $\pi_\tau(C)$ because $\tau$ is tracial. Ultraweak density makes
them equal on $R$. Thus $\pi_\tau(C)$ is contained in the centralizer
\[
 Z_{\widehat\tau}
 =\{x\in R:\widehat\tau(xy)=\widehat\tau(yx)
                         \text{ for every }y\in R\}.
\]
For fixed $y\in R$, the functional
$x\mapsto\widehat\tau(xy)-\widehat\tau(yx)$ is normal, so
$Z_{\widehat\tau}$ is ultraweakly closed. It follows that
$Z_{\widehat\tau}=R$, proving traciality.

The state $\widehat\tau$ is faithful: if
$\widehat\tau(x^*x)=0$, then, for every $y\in R$,
\[
 \begin{aligned}
 \norm{xy\Omega_\tau}^2
 &=\widehat\tau(y^*x^*xy)\\
 &=\widehat\tau(xyy^*x^*)\\
 &\le\norm{y}^2\widehat\tau(xx^*)
  =\norm{y}^2\widehat\tau(x^*x)=0.
 \end{aligned}
\]
Cyclicity of $\Omega_\tau$ gives $x=0$.
If $z$ were a nontrivial central projection and
$t=\widehat\tau(z)\in(0,1)$, the functional
$b\mapsto t^{-1}\widehat\tau(z\pi_\tau(b))$ would be a tracial state of $C$.
Uniqueness of $\tau$ gives
\[
 \widehat\tau(z\pi_\tau(b))
 =t\tau(b)=t\widehat\tau(\pi_\tau(b))\qquad(b\in C).
\]
The two normal functionals $x\mapsto\widehat\tau(zx)$ and
$x\mapsto t\widehat\tau(x)$ agree on the ultraweakly dense subalgebra
$\pi_\tau(C)$, and hence on all of $R$:
\[
 \widehat\tau(zx)=t\widehat\tau(x)\qquad(x\in R).
\]
Taking $x=z$ gives $t=t^2$, a contradiction. Thus $R$ is a finite
infinite-dimensional factor,
and hence is of type~$\mathrm{II}_1$.
\end{proof}

\begin{proof}[Proof of \cref{thm:main}]
Take $A$ and $j$ from \eqref{eq:algebra}. By
\cref{prop:atomic,thm:capture,cor:simple}, this algebra is unital, simple,
and infinite-dimensional, and all its nonzero irreducible representations
are equivalent. The unique trace and the stronger state-extension
assertion are \cref{thm:trace-extension}. By \cref{prop:trace-cyclic}, its
GNS space is separable; its representation is faithful by simplicity.
All choices and arguments above are in ZFC. No diamond principle or
instance of CH is used.
\end{proof}

\section{Consequences}\label{sec:consequences}

\begin{corollary}\label{cor:no-separable-irreducible}
The algebra $A$ has no nonzero irreducible representation on a separable
Hilbert space. Its faithful tracial GNS representation has no nonzero
irreducible closed subrepresentation.
\end{corollary}
\begin{proof}
A separable irreducible representation $\rho:A\to\B(K)$ would, by
\cref{thm:capture}, be faithful. Precomposition with the isomorphism
$\rho:A\to\rho(A)$ identifies the irreducible representations of $\rho(A)$
with those of $A$, preserving unitary equivalence. Thus $\rho(A)$ would
act irreducibly on a separable $K$ and have a unique irreducible class.
Rosenberg's theorem \cite[Theorem~4, p.~530]{Rosenberg} would give
$\rho(A)=\K(K)$, contrary to \cref{cor:simple}. A closed subspace of a separable
Hilbert space is separable, proving the second assertion.
\end{proof}

Write $\mathfrak c=2^{\aleph_0}$, and let $\operatorname{dens}(X)$ denote
the density character of a topological space $X$.

\begin{corollary}\label{cor:cardinal}
For the constructed algebra $A$, its faithful tracial GNS space $H_\tau$,
and the Hilbert space $H$ of the distinguished irreducible representation,
\[
 \begin{gathered}
 |A|=\operatorname{dens}(A)=\mathfrak c,\qquad
 \operatorname{dens}(H)=\mathfrak c,\\
 H_\tau\text{ is separable and infinite-dimensional}.
 \end{gathered}
\]
Consequently, $H_\tau$ has a countably infinite orthonormal basis, whereas
every orthonormal basis of $H$ has cardinality $\mathfrak c$.
\end{corollary}
\begin{proof}
\emph{The algebra and the tracial representation.}
The scalar multiples of $1$ give $\mathfrak c\le|A|$. An operator on the separable space $H_\tau$ is determined by countably
many complex matrix coefficients. Hence $|\B(H_\tau)|\le\mathfrak c$;
the scalar operators give the reverse inequality. The faithful
representation therefore gives
\[
 \mathfrak c\le |A|\le |\B(H_\tau)|=\mathfrak c.
\]
Every norm-dense subset of $A$ generates $A$. The generator lower bound
of \cite[Proposition~6]{AW2004} therefore gives
\[
 \mathfrak c\le\operatorname{dens}(A)\le |A|=\mathfrak c.
\]
The space $H_\tau$ is separable by \cref{prop:trace-cyclic}, and it is
infinite-dimensional because its representation of $A$ is faithful.

\emph{The irreducible representation.}
For the lower bound on $\operatorname{dens}(H)$, we use the pure-state
eigenvector argument underlying \cite[Corollary~5.5.6]{Farah2019}.
Write $e_{11}$ for the second diagonal matrix unit of $M_2(\C)$.
In the diagonal CAR subalgebra $C(\{0,1\}^{\N})$, put
\[
 d=\sum_{n=1}^{\infty}2\cdot3^{-n}
       (1_2^{\otimes(n-1)}\otimes e_{11})\in C,
\]
a norm-convergent series. Its spectrum is the Cantor set
\[
 \operatorname{sp}(d)=
 \left\{\sum_{n=1}^{\infty}2\varepsilon_n3^{-n}:
                    (\varepsilon_n)\in\{0,1\}^{\N}\right\},
 \qquad |\operatorname{sp}(d)|=\mathfrak c.
\]
For each $t\in\operatorname{sp}(d)$, evaluation at $t$ on
$C^*(1,j(d))$ has a pure state extension $\psi_t$ to $A$
\cite[Lemma~5.4.1]{Farah2019}. Indeed, the states extending a pure state
form a nonempty weak$^*$-compact face of the state space; any extreme point
of this face is pure. In its GNS triple
$(K_t,\rho_t,\Omega_t)$,
\[
 \norm{(\rho_t(j(d))-t1)\Omega_t}^2
 =\psi_t((j(d)-t1)^2)=0.
\]
By \cref{thm:capture}, a unitary $U_t:K_t\to H$ intertwines $\rho_t$
with the inclusion. Thus $v_t=U_t\Omega_t$ satisfies
\[
 \norm{v_t}=1,\qquad j(d)v_t=tv_t.
\]
Self-adjointness of $j(d)$ gives
\[
 (t-s)\ip{v_t}{v_s}=0,\qquad
 \norm{v_t-v_s}=\sqrt2\quad(t\ne s).
\]
This uniformly separated family has cardinality $\mathfrak c$, so
$\operatorname{dens}(H)\ge\mathfrak c$. Conversely, any nonzero vector
in the irreducible representation on $H$ is cyclic. The orbit of a
norm-dense subset of $A$ of cardinality $\mathfrak c$ is therefore dense
in $H$, giving $\operatorname{dens}(H)\le\mathfrak c$.
For an infinite-dimensional Hilbert space, the density character equals
the cardinality of an orthonormal basis, proving the final assertions.
\end{proof}

\begin{corollary}\label{cor:ch}
Over ZFC, CH is equivalent to the existence of a counterexample to
Naimark's problem of norm density $\aleph_1$.
\end{corollary}
\begin{proof}
Under CH, \cref{cor:cardinal} supplies such a counterexample. Conversely,
any such algebra $B$ satisfies
\[
 \mathfrak c\le\operatorname{dens}(B)=\aleph_1
\]
by the density bound in \cite[Proposition~6]{AW2004}. Hence CH holds.
\end{proof}

\section*{Formalization}
The associated Lean~4 development, based on mathlib, is available in
MathlibAnnex v0.4.0 \cite{MathlibAnnex040}, with entry module
\nolinkurl{MathlibAnnex.Projects.Naimark}. The citation fixes the source
commit, not a moving branch. The companion project page gives the source
manifest, dependencies, and links to the declarations
\cite{NaimarkFormalizationProject}.

The released statements cover the fixed CAR-based algebra and its
classification of arbitrary nonzero irreducible representations
(\cref{thm:capture,cor:simple}); uniqueness of the extension of the CAR
trace among all states, traciality, faithfulness, and a faithful separable
model (\cref{thm:main,prop:trace-cyclic,thm:trace-extension}); the exclusion
of separable irreducible representations (\cref{cor:no-separable-irreducible});
and the cardinality, density, and CH statements
(\cref{cor:cardinal,cor:ch}). These are result-level correspondences,
not a line-by-line formalization of this article's proofs. The conclusions
about orthonormal-basis cardinalities in \cref{cor:cardinal} are standard
Hilbert-space consequences of the formal density statements, rather than
dedicated endpoints of the cited release. The weak-closure assertion in
\cref{cor:weak-closure} is proved in this article and is not included in
this claim of formalized scope.

The homogeneity declarations cover pure states with the same GNS kernel
on an arbitrary separable unital $C^*$-algebra and produce an
asymptotically inner automorphism implemented by a continuous unitary path
starting at $1$. For CAR, our proof uses only the simple-algebra,
finite-set approximately inner consequence in
\eqref{eq:approximate-inner}.\footnote{The relevant declarations are
\nolinkurl{exists_asymptoticallyInner_of_isPureState_of_sameGNSKernel}
and \nolinkurl{exists_approximatelyInner_of_isPureState}, in the namespace
\nolinkurl{MathlibAnnex.Analysis.CStarAlgebra} and module
\href{https://github.com/r-tanaka-math/mathlib-annex/blob/437e6e46228bbb8e91211ebded349d7a30020e73/MathlibAnnex/Analysis/CStarAlgebra/PureStateHomogeneity/Consequences.lean}{\nolinkurl{MathlibAnnex.Analysis.CStarAlgebra.PureStateHomogeneity.Consequences}}.}
This specifies the homogeneity scope claimed here, rather than every
result of \cite{KOS}. The homogeneity input and the
representation-theoretic density obstruction are proved in the
development, with reuse of mathlib, not added as analytic axioms.

The source manifest resolves the full project dependency closure. The
release records a Lean~4.32.0-rc1 build and compiled-axiom checks with
allowlist \texttt{propext}, \texttt{Classical.choice}, and
\texttt{Quot.sound}, and reports no violations. These checks concern the
released Lean declarations; they do not certify every natural-language
step of the present manuscript. The CH result is a mathematical
biconditional with CH expressed as a cardinal equality, not a forcing or
metatheoretic independence certificate.

\section*{AI use statement}
The new mathematical construction and proof arguments in this paper were
developed predominantly by OpenAI ChatGPT through an iterative process
supervised and reviewed by the author. Building on the classical results
cited in the text, ChatGPT contributed the CAR-based construction and
shell--defect mechanism in
Sections~\ref{sec:flags}--\ref{sec:construction}, in particular
\cref{lem:shell-calculus}; the classification of all irreducible
representations in \cref{thm:capture}; the state-extension and faithful
separable representation arguments in Section~\ref{sec:tracial}; and the
density and CH consequences in Section~\ref{sec:consequences}.
The weak-closure consequence in \cref{cor:weak-closure} was also proposed
during AI-assisted review. ChatGPT assisted with literature discovery,
drafting, and revision of the exposition. ChatGPT and Codex-based tools
were used to propose and develop Lean proof source and to revise candidate
code and proof obligations in the formalization workflow.

The author's contributions were to examine the arguments for mathematical
correctness and to propose improvements to their clarity and organization,
also with AI assistance. The author takes full responsibility for the
content of this paper, including its mathematical correctness, its final
presentation, and the integrity and accuracy of its citations.


\begin{thebibliography}{99}

\bibitem{AAP1986}
C.~A.~Akemann, J.~Anderson, and G.~K.~Pedersen,
\emph{Excising states of $C^*$-algebras},
Canad. J. Math. \textbf{38} (1986), no.~5, 1239--1260.
\href{https://doi.org/10.4153/CJM-1986-063-7}{doi:10.4153/CJM-1986-063-7}.

\bibitem{AW2004}
C.~A.~Akemann and N.~Weaver,
\emph{Consistency of a counterexample to Naimark's problem},
Proc. Natl. Acad. Sci. USA \textbf{101} (2004), no.~20, 7522--7525.
\href{https://doi.org/10.1073/pnas.0401489101}{doi:10.1073/pnas.0401489101}.

\bibitem{CF2023}
D.~Calder\'on and I.~Farah,
\emph{Can you take Akemann--Weaver's $\diamondsuit_{\aleph_1}$ away?},
J. Funct. Anal. \textbf{285} (2023), no.~5, article 110017.
\href{https://doi.org/10.1016/j.jfa.2023.110017}{doi:10.1016/j.jfa.2023.110017}.

\bibitem{Farah2019}
I.~Farah, \emph{Combinatorial Set Theory of $C^*$-algebras},
Springer Monographs in Mathematics, first ed., Springer, Cham, 2019.
\href{https://doi.org/10.1007/978-3-030-27093-3}{doi:10.1007/978-3-030-27093-3}.

\bibitem{FH2017}
I.~Farah and I.~Hirshberg,
\emph{Simple nuclear $C^*$-algebras not isomorphic to their opposites},
Proc. Natl. Acad. Sci. USA \textbf{114} (2017), no.~24, 6244--6249.
\href{https://doi.org/10.1073/pnas.1619936114}{doi:10.1073/pnas.1619936114}.

\bibitem{Glimm1961}
J.~Glimm, \emph{Type~I $C^*$-algebras},
Ann. of Math. (2) \textbf{73} (1961), 572--612.
\href{https://doi.org/10.2307/1970319}{doi:10.2307/1970319}.

\bibitem{KOS}
A.~Kishimoto, N.~Ozawa, and S.~Sakai,
\emph{Homogeneity of the pure state space of a separable $C^*$-algebra},
Canad. Math. Bull. \textbf{46} (2003), no.~3, 365--372.
\href{https://doi.org/10.4153/CMB-2003-038-3}{doi:10.4153/CMB-2003-038-3}.

\bibitem{MathlibAnnex040}
\emph{MathlibAnnex}, version~0.4.0, source release, September~20, 2026.
Entry module: \nolinkurl{MathlibAnnex.Projects.Naimark}.
Commit \href{https://github.com/r-tanaka-math/mathlib-annex/tree/437e6e46228bbb8e91211ebded349d7a30020e73}{\nolinkurl{437e6e46228bbb8e91211ebded349d7a30020e73}}.
\url{https://github.com/r-tanaka-math/mathlib-annex/releases/tag/v0.4.0}.

\bibitem{Naimark1951}
M.~A.~Naimark, \emph{On a problem of the theory of rings with involution},
Uspekhi Mat. Nauk (N.S.) \textbf{6} (1951), no.~6(46), 160--164 (Russian).

\bibitem{NaimarkBrief2026}
\emph{Naimark's problem},
Exact Mathematics with AI: Brief Report,
Draft~R1, September~21, 2026.
Publication responsibility: R.~Tanaka; no mathematical author designated.
\url{https://exactmathematics.org/reports/naimark-problem/draft-r1/naimark-brief-report.pdf}.

\bibitem{NaimarkFormalizationProject}
\emph{Naimark's problem}, MathlibAnnex companion project,
Exact Mathematics with AI, companion to source version~0.4.0, 2026.
\url{https://exactmathematics.org/mathlibannex/projects/naimark-problem/}.

\bibitem{Rosenberg}
A.~Rosenberg,
\emph{The number of irreducible representations of simple rings with no minimal ideals},
Amer. J. Math. \textbf{75} (1953), no.~3, 523--530.
\href{https://doi.org/10.2307/2372501}{doi:10.2307/2372501}.

\bibitem{Tanaka2025}
R.~Tanaka,
\emph{Banach space theoretical construction of (primitive) spectra of
$C^*$-algebras and the Naimark problem revisited},
preprint, arXiv:2504.05551v1, April~7, 2025.
\url{https://arxiv.org/abs/2504.05551v1}.
\href{https://doi.org/10.48550/arXiv.2504.05551}{doi:10.48550/arXiv.2504.05551}.

\bibitem{Vaccaro2018}
A.~Vaccaro, \emph{Trace spaces of counterexamples to Naimark's problem},
J. Funct. Anal. \textbf{275} (2018), no.~10, 2794--2816.
\href{https://doi.org/10.1016/j.jfa.2018.06.018}{doi:10.1016/j.jfa.2018.06.018}.
See also the corrigendum \cite{Vaccaro2020}.

\bibitem{Vaccaro2020}
A.~Vaccaro,
\emph{Corrigendum to ``Trace spaces of counterexamples to Naimark's problem''},
J. Funct. Anal. \textbf{278} (2020), no.~12, article 108458.
\href{https://doi.org/10.1016/j.jfa.2020.108458}{doi:10.1016/j.jfa.2020.108458}.

\bibitem{Weaver2007}
N.~Weaver, \emph{Set theory and $C^*$-algebras},
Bull. Symbolic Logic \textbf{13} (2007), no.~1, 1--20.
\href{https://doi.org/10.2178/bsl/1174668215}{doi:10.2178/bsl/1174668215}.

\end{thebibliography}
\end{document}